\documentclass[final,leqno]{siamltex704}
\usepackage{amsmath}
\usepackage{graphicx}
\usepackage[notcite,notref]{showkeys}
\usepackage{mathrsfs}
\usepackage{graphics}
\usepackage{float}
\usepackage{amsfonts,amssymb}
\usepackage{dsfont}
\usepackage{pifont}
\usepackage{wrapfig} % Allows in-line images
\usepackage{hyperref}
\usepackage{multirow}
\usepackage{color}
\numberwithin{equation}{section}
\usepackage{threeparttable}
\def\3bar{{|\hspace{-.02in}|\hspace{-.02in}|}}
\def\E{{\mathcal{E}}}
\def\T{{\mathcal{T}}}

\def\pT{{\partial T}}

\def\W{{\mathcal{W}}}

\def\bw{{\mathbf{w}}}
\def\bu{{\mathbf{u}}}
\def\bv{{\mathbf{v}}}

\def\bn{{\mathbf{n}}}

\def\bq{{\mathbf{q}}}

\def\be{{\mathbf{e}}}
\def\bw{{\mathbf{w}}}
\def\bf{{\mathbf{f}}}
\def\bQ{{\mathbf{Q}}}
\def\bphi{{\boldsymbol{\phi}}}

\def\bphi{{\boldsymbol{\phi}}}
\def\ljump{{[\![}}
\def\rjump{{]\!]}}

\newtheorem{algorithm}{Weak Galerkin Algorithm}[section]

\title{Weak Galerkin Finite Element Methods for Quad-Curl Problems}

\begin{document}
 \author{
 Chunmei Wang \thanks{Department of Mathematics, University of Florida, Gainesville, FL 32611, USA (chunmei.wang@ufl.edu). The research of Chunmei Wang was partially supported by National Science Foundation Grants DMS-2136380 and DMS-2206332.}
 \and
 Junping Wang\thanks{Division of Mathematical
 Sciences, National Science Foundation, Alexandria, VA 22314
 (jwang@nsf.gov). The research of Junping Wang was supported in part by the
 NSF IR/D program, while working at National Science Foundation.
 However, any opinion, finding, and conclusions or recommendations
 expressed in this material are those of the author and do not
 necessarily reflect the views of the National Science Foundation.}
  \and
 Shangyou Zhang\thanks{Department of Mathematical Sciences,  University of Delaware, Newark, DE 19716, USA (szhang@udel.edu).  }
 }

\maketitle
\begin{abstract}
 This article introduces a weak Galerkin (WG) finite element method for quad-curl problems in three dimensions. It is proved that the proposed WG method is stable and accurate in an optimal order of  error estimates for the exact solution in discrete norms. In addition, an $L^2$ error estimate in an optimal order except the lowest orders $k=1, 2$ is derived for the WG solution. Some numerical experiments are conducted to verify the efficiency and accuracy of our WG method and furthermore a superconvergence has been observed from the numerical results.
\end{abstract}



\pagestyle{myheadings}

\section{Introduction}

In this paper we are concerned with the development of a weak Galerkin (WG) finite element method for the quad-curl problem in three dimensions which seeks $\bu$ such that
\begin{equation}\label{model} 
\begin{split}
 (\nabla \times)^4 \bu=&\bf, \qquad \text{in}\quad \Omega,\\
 \nabla\cdot\bu=&0, \qquad \text{in}\quad \Omega,\\
 \bu\times\bn=&0, \qquad \text{on}\quad \partial\Omega,\\
\nabla\times\bu\times\bn=&0, \qquad \text{on}\quad \partial\Omega,
 \end{split}
\end{equation}
for a given $\bf$ defined on a bounded domain $\Omega\subset \mathbb R^3$. 

The quad-curl problems arise in inverse electromagnetic scattering theory for nonhomogeneous
media \cite{13} and magneto-hydrodynamics equations \cite{55}.
Recently, some contributions have been made on the finite element methods for the quad-curl problems.  The conforming finite element spaces for the quad-curl problem have been recently constructed in  two dimensions (e.g. \cite{28,51}) and    in three dimensions (e.g. \cite{27,39,52}). \cite{30,55} proposed the nonconforming and low order finite element spaces for the quad-curl problems. \cite{sun2,49,53} proposed the mixed methods for the quad-curl problems. \cite{10} introduced a formulation using the Hodge decomposition for the quad-curl problems. \cite{25} introduced a discontinuous Galerkin scheme.   \cite{sun} proposed a novel weak Galerkin formulation using the conforming space for curl-curl problem as a nonconforming space for the quad-curl problem. \cite{50} analyzed a posteriori error analysis for the quad-curl problems in two dimensions.  \cite{54} introduced a virtual element method for the quad-curl problems in two dimensions. \cite{cao} introduced a decoupled formulation for the quad-curl problems where  the a priori and a posteriori error were analyzed. 

  In the literature, the existing WG methods for quad-curl problems  proposed  in \cite{sun}  were curl-conforming  and  based on tetrahedral partitions. However, our WG method is not necessary to be curl-conforming and is based on any polyhedral partitions.   Our WG numerical method  (\ref{32})-(\ref{2})   has provided an accurate and reliable numerical solution for the quad-curl system (\ref{model}) in an optimal order of error estimates in discrete norms and in an optimal order of $L^2$ error estimates except the lowest two orders $k=1, 2$. In addition, we have observed some superconvergence phenomena from numerical experiments. 

The paper is organized as follows. Section 2 is devoted to the derivation of a weak formulation for the quad-curl system \eqref{model}. Section 3  briefly introduces the discrete weak gradient operator and  the discrete weak curl-curl  operator. Section 4 is dedicated to the presentation of the weak Galerkin  algorithm for the quad-curl problem and a discussion of the solution existence and uniqueness for the WG scheme. In Section 5, the error equations are derived for the WG scheme. Section 6 establishes an optimal order of error estimates in discrete norms for  the WG approximation. In Section 7, the $L^2$ error estimate for the WG solution is established in an optimal order except the lowest two orders $k=1, 2$ under some regularity assumptions. Section 8 demonstrates the numerical performance of the WG algorithm through some test examples. 

We follow the standard notations for Sobolev spaces and norms defined on a given open and bounded domain $D\subset \mathbb{R}^3$ with Lipschitz continuous boundary. Denote by $\|\cdot\|_{s,D}$, $|\cdot|_{s,D}$ and $(\cdot,\cdot)_{s,D}$ the norm, seminorm and inner product in the Sobolev space $H^s(D)$ for any $s\ge 0$. The space $H^0(D)$ coincides with $L^2(D)$ (i.e., the space of square integrable functions), for which the norm and the inner product are denoted by $\|\cdot \|_{D}$ and $(\cdot,\cdot)_{D}$. When $D=\Omega$ or when the domain of integration is clear from the context, we shall drop the subscript $D$ in the norm and the inner product notation.
%For simplicity, we shall use ``$\lesssim$'' to denote ``less than or equal to'' up to a generic constant independent of the meshsize and other physical or functional parameters.

\section{A Weak Formulation}\label{Section:2}

Let $s>0$ be an integer. We first introduce
$$
H(curl^s; \Omega)=\{\bu\in [L^2(\Omega)]^3: (\nabla\times)^j\bu\in [L^2(\Omega)]^3, j=1, \cdots, s\} 
$$
with the associated inner product 
  $
 (\bu, \bv)_{H(curl^s; \Omega)}=(\bu, \bv)+\sum_{j=1}^s ((\nabla\times)^j\bu,(\nabla\times)^j\bv)
  $
  and the norm $
 \|\bu\|_{H(curl^s; \Omega)}= (\bu, \bu)^{\frac{1}{2}}_{H(curl^s; \Omega)}$. 
 We  further introduce 
$$
H_0(curl; \Omega):=\{\bu\in H(curl; \Omega): \bn\times \bu=0\ \text{on}\  \partial\Omega\},
$$
$$
H_0(curl^2; \Omega):=\{\bu\in H(curl^2; \Omega): \bn\times \bu=0\ \text{and}\ \nabla\times\bu\times \bn=0\ \text{on}\ \partial \Omega\}.
$$
We   introduce 
$$
H(div; \Omega)=\{\bu\in [L^2(\Omega)]^3: \nabla\cdot\bu \in L^2(\Omega)\},
$$
with the associated inner product $(\bu, \bv)_{H(div; \Omega)}=(\bu, \bv)+(\nabla\cdot\bu, \nabla\cdot\bv)$ and the norm $\|\bu\|_{H(div; \Omega)}= (\bu, \bu)^{\frac{1}{2}}_{H(div; \Omega)} $. We further introduce
$$
H(div^0; \Omega)=\{\bu\in H(div;\Omega): \nabla\cdot\bu=0\ 
 \text{in} \ \Omega\}.
$$

Using the usual integration by parts, we are ready to propose the weak formulation of the quad-curl problem \eqref{model} as follows: Given $\bf\in H(div^0; \Omega)$, find $(\bu; p)\in H_0(curl^2; \Omega)\times H_0^1(\Omega)$ such that
\begin{equation}\label{weakform}
\begin{split}
((\nabla\times)^2\bu, (\nabla\times)^2\bv)+(\bv, \nabla p)=&(\bf, \bv), \qquad \bv \in H_0(curl^2; \Omega),\\
-(\bu, \nabla q)=&0, \qquad \qquad \forall q\in H_0^1(\Omega).
\end{split}
\end{equation}

\begin{theorem} \cite{sun}
Given $\bf\in H(div^0; \Omega)$, the problem \eqref{weakform} has a unique solution $(\bu; p)\in H_0(curl^2; \Omega)\times H_0^1(\Omega)$. Furthermore, $p=0$ and $\bu$ satisfies
$$
\|\bu\|_{H(curl^2; \Omega)}\leq C\|\bf\|.
$$

\end{theorem}

%The paper is organized as follows. In Section \ref{Section-02}, we derive a weak formulation for the elliptic interface problem \eqref{model}-\eqref{model2} that is derivative-free on the solution variable. In Section \ref{Section:Hessian}, we briefly review the weak differential operators and their discrete analogies. In Section \ref{Section:WGFEM}, we describe the PDWG method for the model problem \eqref{model} based on the weak formulation (\ref{weakform}) and its dual. In Section \ref{Section:ExistenceUniqueness}, we establish the solution existence, uniqueness, and stability. In Section \ref{Section:error-equation}, we provide an error equation for the PDWG solutions. In Section \ref{Section:error-estimates}, we derive some error estimates based on various regularity assumptions on the exact solution. Finally, in Section \ref{Section:NE} we report a couple of numerical results to illustrated and verify our convergence theory.

\section{Weak Differential Operators}\label{Section:Hessian}
The principal differential operators in the weak formulation (\ref{weakform}) for the quad-curl problem (\ref{model}) are the gradient operator $\nabla$ and the curl-curl operator $(\nabla \times)^2$. We shall briefly review the discrete weak gradient operator \cite{wwconvdiff, wy3655} and define the discrete weak curl-curl operator.

Let $T$ be a polyhedral domain with boundary $\partial T$. A scalar-valued weak function on $T$ refers to $\sigma=\{\sigma_0,\sigma_b\}$ with $\sigma_0\in L^2(T)$ and $\sigma_b\in L^{2}(\partial T)$. Here $\sigma_0$ and $\sigma_b$ are used to represent the value of $\sigma$ in the interior and on the boundary of $T$. Note that $\sigma_b$  may not necessarily be the trace of $\sigma_0$   on $\partial T$. Denote by $\W(T)$ the space of scalar-valued weak functions on $T$:
\begin{equation}\label{2.1}
\W(T)=\{\sigma=\{\sigma_0,\sigma_b\}: \sigma_0\in L^2(T), \sigma_b\in
L^{2}(\partial T)\}.
\end{equation}
A vector-valued weak function on $T$ refers to a triplet $\bv=\{\bv_0,\bv_b, \bv_n\}$  where  $\bv_0$ and $\bv_b$ are used to represent the values of $\bv$ in the interior and on the boundary of $T$ and $\bv_n$ represents the value of $\nabla\times\bv$ on $\partial T$. Note that $\bv_b$ and $\bv_n$ may not necessarily be the traces of $\bv_0$ and $\nabla \times \bv_0$ on $\partial T$ respectively. Denote by $V(T)$ the space of vector-valued weak functions on $T$:
\begin{equation}
V(T)=\{\bv=\{\bv_0,\bv_b, \bv_n\}: \bv_0\in [L^2(T)]^3,  \bv_b\in
[L^{2}(\partial T)]^3, \bv_n\in [L^{2}(\partial T)]^3\}.
\end{equation}
 
The weak gradient of $\sigma\in \W(T)$, denoted by $\nabla_w \sigma$, is defined as a linear functional on $[H^1(T)]^3$ such that
\begin{equation*}
(\nabla_w  \sigma,\boldsymbol{\psi})_T=-(\sigma_0,\nabla \cdot \boldsymbol{\psi})_T+\langle \sigma_b,\boldsymbol{\psi}\cdot \textbf{n}\rangle_{\partial T},
\end{equation*}
for all $\boldsymbol{\psi}\in [H^1(T)]^3$.

The weak curl-curl operator of any $\bv\in V(T)$, denoted by $(\nabla\times)^2_{w}\bv$ is defined in the dual space of $H(curl^2; T)$, whose action on $\bq\in H(curl^2; T)$ is given by 
$$
((\nabla \times)^2_{w} \bv, \bq)_T=(\bv_0, (\nabla\times)^2 \bq)_T-\langle \bv_b\times\bn, \nabla\times\bq\rangle_{\partial T}-\langle \bv_n\times\bn, \bq\rangle_{\partial T}.
$$

Denote by $P_r(T)$ the space of polynomials on $T$ with degree no more than $r$. 

A discrete version of $\nabla_{w}\sigma$  for $\sigma\in \W(T)$, denoted by $\nabla_{w, r, T}\sigma$, is defined as a unique polynomial vector in $[P_r(T) ]^3$ satisfying
\begin{equation}\label{disgradient}
(\nabla_{w, r, T} \sigma, \boldsymbol{\psi})_T=-(\sigma_0, \nabla \cdot \boldsymbol{\psi})_T+\langle \sigma_b, \boldsymbol{\psi} \cdot \textbf{n}\rangle_{\partial T}, \quad\forall\boldsymbol{\psi}\in [P_r(T)]^3,
\end{equation}
 which, from the usual integration by parts, gives
 \begin{equation}\label{disgradient*}
 (\nabla_{w, r, T} \sigma, \boldsymbol{\psi})_T= (\nabla \sigma_0, \boldsymbol{\psi})_T-\langle \sigma_0- \sigma_b, \boldsymbol{\psi} \cdot \textbf{n}\rangle_{\partial T}, \quad\forall\boldsymbol{\psi}\in [P_r(T)]^3,
 \end{equation}
 provided that $\sigma_0\in H^1(T)$.

A discrete version of $(\nabla\times)^2_{w}\bv$ for $\bv\in V(T)$, denoted by $(\nabla\times)^2_{w, r, T}\bv$, is defined as a unique polynomial vector in $[P_r(T) ]^3$ satisfying
\begin{equation}\label{discurlcurl}
((\nabla \times)^2_{w, r, T} \bv, \bq)_T=(\bv_0, (\nabla\times)^2 \bq)_T-\langle \bv_b\times\bn, \nabla\times\bq\rangle_{\partial T}-\langle \bv_n\times\bn, \bq\rangle_{\partial T}, 
\end{equation}
 for any $\bq \in [P_r(T)]^3$.

\section{Weak Galerkin Algorithm}\label{Section:WGFEM}
Let ${\cal T}_h$ be a finite element partition of the domain $\Omega\subset\mathbb R^3$ consisting of polyhedra that are shape-regular \cite{wy3655}. Denote by ${\mathcal E}_h$ the set of all faces in ${\cal T}_h$ and  ${\mathcal E}_h^0={\mathcal E}_h \setminus
\partial\Omega$ the set of all interior faces. Denote by $h_T$ the meshsize of $T\in {\cal T}_h$ and $h=\max_{T\in {\cal T}_h}h_T$ the meshsize for the partition ${\cal T}_h$.

For any given integer $k\geq 1$, denote by
$W_k(T)$ the local discrete space of the scalar-valued weak functions given by
$$
W_k(T)=\{\{\sigma_0,\sigma_b\}:\sigma_0\in P_k(T),\sigma_b\in
P_k(e),e\subset \partial T\}.
$$
Furthermore, denote by
$V_k(T)$ the local discrete space of the vector-valued weak functions given by
$$
V_k(T)=\{\{\bv_0,\bv_b, \bv_n\}:\bv_0\in [P_k(T)]^3,  \bv_b\in
[P_k(e)]^3, \bv_n\in
[P_{k-1}(e)]^3, e\subset \partial T\}.
$$
Patching $W_k(T)$ over all the elements $T\in {\cal T}_h$
through a common value $\sigma_b$ on the interior interface $\E_h^0$, we arrive at the following scalar-valued 
weak finite element space, denoted by $W_h$; i.e.,
$$
W_h=\big\{\{\sigma_0, \sigma_b\}:\{\sigma_0, \sigma_b\}|_T\in W_k(T), \forall T\in {\cal T}_h \big\},
$$
and the subspace of $W_h$ with vanishing boundary values on $\partial\Omega$, denoted by $W_h^0$; i.e.,
\begin{equation}\label{W0}
W_h^0=\{\{\sigma_0, \sigma_b\}\in W_h: \sigma_b=0\ \text{on}\ \partial \Omega\}.
\end{equation}
Similarly,  patching $V_k(T)$ over all the elements $T\in {\cal T}_h$
through a common value $\bv_b$ on the interior interface $\E_h^0$, we arrive at the following vector-valued weak finite element space, denoted by $V_h$; i.e.,
$$
V_h=\big\{\{\bv_0,\bv_b, \bv_n\}:\{\bv_0,\bv_b, \bv_n\}|_T\in V_k(T), \forall T\in {\cal T}_h \big\},
$$
and the subspace of $V_h$ with vanishing boundary values on $\partial\Omega$, denoted by $V_h^0$; i.e.,
\begin{equation}\label{V0}
V_h^0=\big\{\{\bv_0,\bv_b, \bv_n\}\in V_h: \bv_b\times\bn=0\  \text{and} \ \bv_n\times\bn=0\ \ \text{on}\ \partial\Omega\big\}.
\end{equation}

For simplicity of notation and without confusion, for any $\sigma\in
W_h$ and $\bv\in V_h$, denote by $\nabla_{w}\sigma$ and $(\nabla \times) ^2_{w} \bv$ the discrete weak actions   $\nabla_{w, k, T}\sigma$ and  $(\nabla \times) ^2_{w, k-2, T} \bv$ computed by using   (\ref{disgradient}) and \eqref{discurlcurl} on each element $T$; i.e.,
$$
(\nabla_{w}\sigma)|_T= \nabla_{w, k, T}(\sigma|_T), \qquad \sigma\in W_h,
$$
 $$
({\nabla\times}^2)_{w} \bv|_T=({\nabla\times}^2)_{w, k-2, T}(\bv|_T), \qquad \bv\in V_h.
$$

For any $\sigma, \lambda\in W_h$ and $\bu, \bv\in V_h$, we introduce the
following bilinear forms
\begin{align} \label{EQ:local-stabilizer}
a(\bu, \bv)=&\sum_{T\in {\cal T}_h}a(\bu, \bv),\\
b(\bu, \lambda)=&\sum_{T\in {\cal T}_h}b_T(\bu, \lambda), \\
s_1(\bu, \bv)=&\sum_{T\in {\cal T}_h}s_{1,T}(\bu, \bv),\\
s_2(\sigma, \lambda)=&\sum_{T\in {\cal T}_h}s_{2,T}(\sigma, \lambda),
\label{EQ:local-bterm}
\end{align}
where
\begin{equation*}
\begin{split}
a_T(\bu, \bv) =& ((\nabla\times)_w^2 \bu, (\nabla\times)_w^2 \bv)_T,\\
b_T(\bu, \lambda)=&(\bu_0, \nabla_w \lambda)_T,\\
s_{1,T}(\bu, \bv)=&h_T^{-3}\langle \bu_0\times\bn-\bu_b\times\bn, \bv_0\times\bn-\bv_b\times\bn \rangle_{\partial T}\\&+h_T^{-1}\langle \nabla\times\bu_0\times\bn-\bu_n\times\bn, \nabla\times\bv_0\times\bn-\bv_n\times\bn\rangle_{\partial T},\\
s_{2,T}(\sigma, \lambda)=&h_T^3 \langle \sigma_0-\sigma_b, \lambda_0-\lambda_b\rangle_{\partial T}.  
 \end{split}
\end{equation*}

The following is the weak Galerkin scheme for the quad-curl problem (\ref{model}) based on the variational formulation (\ref{weakform}).
\begin{algorithm}\label{a-1}
Given $\bf \in H(div^0; \Omega)$, find $(\bu_h; p_h)\in V_h^0 \times W_{h}^0$, such that
\begin{eqnarray}\label{32}
s_1(\bu_h, \bv_h)+a(\bu_h, \bv_h)+b(\bv_h, p_h)&=& (\bf, \bv_0), \qquad \forall \bv_h\in V_{h}^0,\\
s_2(p_h, q_h)-b(\bu_h, q_h)&=&0,\qquad \quad \qquad \forall q_h\in W_h^0.\label{2}
\end{eqnarray}
\end{algorithm} 
 
\begin{theorem}
The weak Galerkin finite element scheme (\ref{32})-(\ref{2}) has a unique solution.
\end{theorem}
\begin{proof}
It suffices to prove that $\bf=0$ implies that $\bu_h=0$ and $p_h=0$ in $\Omega$. 
To this end, taking $\bv_h=\bu_h$ in (\ref{32}) and $q_h=p_h$ in (\ref{2}) gives
$$
((\nabla\times)^2_w\bu_h, (\nabla\times)^2_w\bu_h)+s_1(\bu_h, \bu_h)+s_2(p_h, p_h)=0.
$$
This yields
\begin{eqnarray} 
(\nabla\times)^2_w\bu_h&=&0, \quad \text{in each}\ T,\label{t1}\\
\nabla \times \bu_0\times \bn &=&\bu_n\times \bn,   \quad \text{on each}\ \partial T,\label{t2}\\
 \bu_0 \times \bn&=&\bu_b\times \bn, \quad \text{on each}\ \partial T,\label{t3}\\
 p_0&=&p_b, \quad \text{on each}\ \partial T.\label{t4}
 \end{eqnarray}
Using \eqref{t1}, \eqref{discurlcurl}, \eqref{t2}-\eqref{t3}, and the integration by parts, we obtain
\begin{equation*}
\begin{split}
0=&((\nabla\times)^2_w\bu_h, \bw)_T\\
=&((\nabla\times)^2\bu_0, \bw)_T-\langle \bw, (\bu_n-\nabla\times\bu_0)\times\bn\rangle_{\pT}+\langle \nabla\times \bw, (\bu_0-\bu_b)\times\bn\rangle_{\pT}\\
=&((\nabla\times)^2\bu_0, \bw)_T,
\end{split}
\end{equation*}
for any $\bw\in [P_{k-2}(T)]^3$. This gives
$(\nabla\times)^2\bu_0=0$ in each $T\in \T_h$. It follows from \eqref{t2}-\eqref{t3} that $\bu_0\times\bn$ and $\nabla\times \bu_0\times \bn$ are continuous across the interior interface ${\cal E}_h^0$. Thus, $\bu_0\in H(curl^2; \Omega)$ and $(\nabla \times)^2 \bu_0=0$ in $\Omega$. Therefore, there exists a potential function $\phi$ such that  $\nabla \times \bu_0=\nabla \phi$ in $\Omega$. This gives
\begin{equation}
    \begin{split}
  (\nabla \phi, \nabla \phi)  &= \sum_{T\in {\cal T}_h} (\nabla \times \bu_0, \nabla \phi)_T\\
   &= \sum_{T\in {\cal T}_h} ( \bu_0, \nabla \times\nabla \phi)_T +\langle \nabla \phi, \bn\times \bu_0\rangle_{\partial T} 
    \\&= \sum_{T\in {\cal T}_h} \langle \nabla \phi, \bn\times \bu_b\rangle_{\partial T}\\
   &=  \langle \nabla \phi, \bn\times \bu_b\rangle_{\partial \Omega}\\&=0,
   \end{split}
\end{equation}
where we used the usual integration by parts, \eqref{t3} and $\bn\times \bu_b=0$ on $\partial\Omega$.
% Since $\nabla\times \bu_0\in H(div;\Omega)$, we have $\Delta \phi=\nabla\cdot\nabla\times \bu_0=0$ in $\Omega$, which together with the boundary condition $\nabla \phi\times \bn=0$ on $\partial\Omega$ 
This leads to $\phi=C$ in $\Omega$, and thus $\nabla \times \bu_0=0$ in $\Omega$.  Furthermore, there exists a potential function $\psi$ such that $\bu_0=\nabla \psi$ in $\Omega$. 

From \eqref{t4}, \eqref{disgradient} and \eqref{2}, we have
\begin{equation}\label{ee1}
\begin{split}
0&=\sum_{T\in {\cal T}_h} (\nabla_w q_h, \bu_0)_T\\
&=\sum_{T\in {\cal T}_h}-(q_0, \nabla\cdot\bu_0)_T+\langle q_b, \bu_0\cdot\bn\rangle_{\pT}\\
&=\sum_{T\in {\cal T}_h} -(q_0, \nabla\cdot\bu_0)_T+\sum_{e\in {\cal E}_h^0}\langle q_b, \ljump \bu_0\cdot\bn\rjump \rangle_{e},
\end{split}
\end{equation}
where $\ljump \bu_0\cdot\bn\rjump$ is the jump of $\bu_0\cdot\bn$ on edge $e\in {\cal E}_h^0$ and we used $q_b=0$ on $\partial\Omega$.
Letting $q_0=0$ and $q_b=\ljump \bu_0\cdot\bn \rjump$ in \eqref{ee1} yields that $\ljump \bu_0\cdot\bn \rjump=0$ on $e\in {\cal E}_h^0$ which means $\bu_0\cdot\bn$ is continuous along the interior interface $e\in {\cal E}_h^0$. This follows that $\bu_0\in H(div; \Omega)$. Taking $q_0=\nabla\cdot \bu_0$ and $q_b=0$  in \eqref{ee1} gives $\nabla \cdot \bu_0=0$ on each $T$ and further $\nabla \cdot \bu_0=0$ in $\Omega$ due to $\bu_0\in H(div; \Omega)$. Recall that there exists a potential function $\psi$ such that $\bu_0=\nabla \psi$ in $\Omega$.  Hence, $\nabla\cdot\bu_0=\Delta \psi=0$ strongly holds true in $\Omega$ with the boundary condition $\nabla \psi \times \bn=\bu_0\times \bn=0$ on $\partial \Omega$. This implies that $\psi=C$ in $\Omega$. Thus, $\bu_0=\nabla \psi=0$ in $\Omega$. Using \eqref{t2}-\eqref{t3} gives $\bu_b=0$ and $\bu_n=0$ in $\Omega$. Therefore, we obtain $\bu_h=0$ in $\Omega$.

Using $\bu_h=0$ gives $s_1(\bu_h, \bv_h)+a(\bu_h, \bv_h)=0$ for any $\bv_h\in V_h^0$. It follows from the assumption $\bf=0$ and \eqref{32}  that $b(\bv_h, p_h)=0$, which, together with \eqref{t4} and (\ref{disgradient}) and the usual integration by parts,  gives
$$
0=b(\bv_h, p_h)=-\sum_{T\in {\cal T}_h} (p_0, \nabla\cdot\bv_0)_T+\langle p_b, \bv_0\cdot\bn\rangle_{\pT} = \sum_{T\in {\cal T}_h} (\nabla p_0,  \bv_0)_T.
$$
Letting $\bv_0=\nabla p_0$ gives rise to $\nabla p_0=0$ on each $T\in {\cal T}_h$; i.e., $p_0=C$ on each $T\in {\cal T}_h$. The facts that $p_0=p_b$ on each $\pT$ and $p_b=0$ on $\partial\Omega$ give $p_0=p_b=0$ in $\Omega$ and further $p_h=0$ in $\Omega$.

This completes the proof of the theorem. 
\end{proof}

%\section{Stability Analysis}\label{Section:ExistenceUniqueness}
Let $k\geq 1$. Let $\bQ_0$ be the $L^2$ projection operator onto $[P_k(T)]^3$. Analogously, for $e\subset\partial T$, denote by $\bQ_b$ and $\bQ_n$ the $L^2$ projection operators onto $[P_{k}(e)]^3$ and $[P_{k-1}(e)]^3$, respectively. For $\bw\in [H(curl; \Omega)]^3$, define the $L^2$ projection $\bQ_h \bw\in V_h$ as follows
$$
\bQ_h\bw|_T=\{\bQ_0 \bw, \bQ_b \bw, \bQ_n(\nabla\times\bw)\}.
$$
For $\sigma \in H^1(\Omega)$, the $L^2$ projection $Q_h \sigma\in W_h$ is defined by 
 $$
 Q_h\sigma|_T=\{Q_0 \sigma, Q_b \sigma\},
$$
where $Q_0$ and $Q_b$ are the $L^2$ projection operators onto $P_k(T)$ and $P_k(e)$ respectively.  Denote by ${\cal Q}_h^{k-2}$ and ${\cal Q}_h^{k}$  the $L^2$ projection operators onto $P_{k-2}(T)$ and $P_{k}(T)$, respectively.

\begin{lemma}\label{Lemma5.1}   The operators $\bQ_h$, $Q_h$, ${\cal Q}^{k}_h$ and ${\cal Q}^{k-2}_h$ satisfy the following commutative properties:
\begin{eqnarray}\label{l}
(\nabla\times)^2_w(\bQ_h \bw) &=& {\cal Q}_h^{k-2}((\nabla\times)^2 \bw), \qquad  \forall \bw\in H(curl^2; T),\\
\nabla_{w}(Q_h \sigma) &=& {\cal Q}^{k}_h(\nabla \sigma),  \qquad \qquad \forall  \sigma\in H^1(T). \label{l-2}
\end{eqnarray}
\end{lemma}
\begin{proof}
For any $\bq \in [P_{k-2}(T)]^3$, using \eqref{discurlcurl} and the usual integration by parts gives
\begin{equation*}
\begin{split}
((\nabla\times)_w^2\bQ_h\bw, \bq)_T=&(\bQ_0\bw, (\nabla\times)^2 \bq)_T-\langle \bQ_b\bw\times\bn, \nabla\times\bq\rangle_{\partial T}-\langle \bQ_n(\nabla\times\bw)\times\bn, \bq\rangle_{\partial T}\\
=&(\bw, (\nabla\times)^2 \bq)_T-\langle\bw\times\bn, \nabla\times\bq\rangle_{\partial T}-\langle \nabla\times \bw\times\bn, \bq\rangle_{\partial T}\\
=&((\nabla\times)^2 \bw, \bq)_T\\
=&({\cal Q}_h^{k-2}((\nabla\times)^2 \bw), \bq)_T.
\end{split}
\end{equation*}
This completes the proof of \eqref{l}. 

The proof of \eqref{l-2} can be found in \cite{wwconvdiff, wy3655}.
\end{proof}

\section{Error Equations}\label{Section:error-equation}
The goal of this section is to derive the error equations for the weak Galerkin method (\ref{32})-(\ref{2}) for solving the quad-curl problem \eqref{model}, which play a critical role in the forthcoming convergence analysis.

 Let $(\bu, p)$ be the solution  of \eqref{weakform} and assume that $\bu\in H(curl^4; \Omega)$. Then $(\bu, p)$  satisfies
 \begin{eqnarray} \label{mo1}
 ((\nabla\times)^4\bu, \bv)+(\bv, \nabla p)&=&(\bf, \bv),\\
 (\nabla\cdot\bu, q)&=&0, \label{mo2}
 \end{eqnarray}
for $\bv \in [L^2(\Omega)]^3$ and $q\in L^2(\Omega)$. Let $(\bu_h, p_h)$ be the WG solutions of \eqref{32}-\eqref{2}. Define the error functions $\be_h$ and $\epsilon_h$ by 
\begin{eqnarray}\label{error}
\be_h&=&\{\be_0, \be_b, \be_n\}=\{\bQ_0\bu-\bu_0, \bQ_b\bu-\bu_b, \bQ_n(\nabla\times\bu)-\bu_n\},\\
\epsilon_h&=&\{\epsilon_0, \epsilon_b\}=\{Q_0p-p_0, Q_bp-p_b\}. \label{error-2}
\end{eqnarray}
 \begin{lemma}\label{errorequa}
Let $\bu\in H(curl^4; \Omega)$ and $(\bu_h; p_h) \in V_h^0\times W_h^0$ be the exact solution of quad-curl model problem \eqref{model} and the numerical solution arising from the WG scheme (\ref{32})-(\ref{2}) respectively. The error functions $\be_h$ and $\epsilon_h$ defined in (\ref{error})-(\ref{error-2}) satisfy the following error equations; i.e., 
\begin{eqnarray}\label{sehv}
s_1(\be_h, \bv_h)+a(\be_h, \bv_h)+b(\bv_h, \epsilon_h)&=& s_1(\bQ_h\bu, \bv_h) +\ell_1(\bu, \bv_h),\quad   \forall \bv_h\in V_{h}^0,\\
-b(\be_h, q_h)+s_2(\epsilon_h, q_h)&=&s_2(Q_hp, q_h)-\ell_2(\bu, q_h),\quad\forall q_h\in W^0_h. \label{sehv2}
\end{eqnarray}
\end{lemma}
Here
 \begin{eqnarray*}\label{lu}
\ell_1(\bu, \bv_h)&=&\sum_{T\in{\cal T}_h}\langle (\bv_0-\bv_b)\times\bn, \nabla\times({\cal Q}_h^{k-2} -I)((\nabla\times)^2\bu)\rangle_{\partial T} \\&&+\langle (\nabla \times \bv_0-\bv_n)\times\bn, ({\cal Q}_h^{k-2} -I)((\nabla\times)^2\bu)\rangle_{\partial T},\\
  \ell_2(\bu, q_h)&=&  \sum_{T\in{\cal T}_h} \langle q_0-q_b, (I-\bQ_0)\bu \cdot\bn\rangle_{\pT}.
\end{eqnarray*}

\begin{proof}
Using \eqref{l}, \eqref{discurlcurl} and the usual integration by parts, we have
\begin{equation}\label{eq1}
\begin{split}
&((\nabla\times)^2_w\bQ_h\bu,  (\nabla\times)^2_w \bv_h)_T\\=& ({\cal Q}_h^{k-2} ((\nabla\times)^2\bu), (\nabla\times)^2_w \bv_h)_T\\
=& (\bv_0, (\nabla\times)^2 {\cal Q}_h^{k-2} ((\nabla\times)^2\bu))_T-\langle \bv_b\times\bn, \nabla\times{\cal Q}_h^{k-2} ((\nabla\times)^2\bu)\rangle_{\partial T}\\&-\langle \bv_n\times\bn, {\cal Q}_h^{k-2} ((\nabla\times)^2\bu)\rangle_{\partial T} \\ 
%=& (\bv_0, (\nabla\times)^2 {\cal Q}_h^{k-1} ((\nabla\times)^2\bu))_T-\langle \bv_b\times\bn, \nabla\times{\cal Q}_h^{k-1} ((\nabla\times)^2\bu)\rangle_{\partial T}\\&-\langle \bv_n\times\bn, {\cal Q}_h^{k-1} ((\nabla\times)^2\bu)\rangle_{\partial T} \\ 
=& ((\nabla\times)^2 \bv_0,  {\cal Q}_h^{k-2} ((\nabla\times)^2\bu))_T+\langle (\bv_0-\bv_b)\times\bn, \nabla\times{\cal Q}_h^{k-2} ((\nabla\times)^2\bu)\rangle_{\partial T}\\&+\langle (\nabla \times \bv_0-\bv_n)\times\bn, {\cal Q}_h^{k-2} ((\nabla\times)^2\bu)\rangle_{\partial T} \\ 
=& ((\nabla\times)^2 \bv_0,   ((\nabla\times)^2\bu))_T+\langle (\bv_0-\bv_b)\times\bn, \nabla\times{\cal Q}_h^{k-2} ((\nabla\times)^2\bu)\rangle_{\partial T}\\&+\langle (\nabla \times \bv_0-\bv_n)\times\bn, {\cal Q}_h^{k-2} ((\nabla\times)^2\bu)\rangle_{\partial T}. 
\end{split}
\end{equation}
Taking $\bv=\bv_0$ in \eqref{mo1} where $\bv_h=\{\bv_0, \bv_b, \bv_n\} \in V_h^0$ and using the usual integration by parts, we get
\begin{equation}\label{eq2}
\begin{split}
&\sum_{T\in {\cal T}_h}((\nabla\times)^2\bu, (\nabla\times)^2\bv_0)_T+\langle (\nabla\times)^3\bu, (\bv_0-\bv_b)\times\bn\rangle_{\pT}\\&+\langle (\nabla\times)^2\bu, \nabla\times\bv_0\times\bn-\bv_n\times\bn\rangle_{\pT}+(\nabla p, \bv_0)_T=\sum_{T\in {\cal T}_h} (\bf, \bv_0)_T,
\end{split}
\end{equation}
where we used the facts that
$$
\sum_{T\in {\cal T}_h} \langle (\nabla\times)^2\bu, \bv_n\times\bn\rangle_{\pT}=\langle (\nabla\times)^2\bu, \bv_n\times\bn\rangle_{\partial\Omega}=0,
$$
$$
\sum_{T\in {\cal T}_h} \langle (\nabla\times)^3\bu, \bv_b\times\bn\rangle_{\pT}=\langle (\nabla\times)^3\bu, \bv_b\times\bn\rangle_{\partial\Omega}=0.
$$
Substituting \eqref{eq2} into \eqref{eq1}  gives
\begin{equation}\label{e3} \begin{aligned} 
& \quad \
((\nabla\times)^2_w\bQ_h\bu,  (\nabla\times)^2_w \bv_h)\\
&=(\bf-\nabla p, \bv_0) +\langle (\bv_0-\bv_b)\times\bn, \nabla\times({\cal Q}_h^{k-2} -I)((\nabla\times)^2\bu)\rangle_{\partial T}\\
&\quad \ +\langle (\nabla \times \bv_0-\bv_n)\times\bn, ({\cal Q}_h^{k-2} -I)((\nabla\times)^2\bu)\rangle_{\partial T}.
\end{aligned}
\end{equation}
It follows from \eqref{l-2} that
\begin{equation}\label{eq4}
\begin{split}
b(\bv_h, Q_hp)=(\nabla_w(Q_hp), \bv_0)=({\cal Q}_h^{k}(\nabla p), \bv_0)=(\nabla p, \bv_0).
\end{split}
\end{equation}
Combining \eqref{e3}-\eqref{eq4} gives
\begin{equation*} 
\begin{split}
&s_1(\bQ_h\bu, \bv_h)+a(\bQ_h\bu, \bv_h)+b(\bv_h, Q_hp)\\
=& (\bf, \bv_0) +\langle (\bv_0-\bv_b)\times\bn, \nabla\times({\cal Q}_h^{k-2} -I)((\nabla\times)^2\bu)\rangle_{\partial T}\\&+\langle (\nabla \times \bv_0-\bv_n)\times\bn, ({\cal Q}_h^{k-2} -I)((\nabla\times)^2\bu)\rangle_{\partial T}+s_1(\bQ_h\bu, \bv_h).
\end{split}
\end{equation*}
Subtracting \eqref{32} from the above equation gives \eqref{sehv}.

To derive \eqref{sehv2}, taking $q=q_0$ in \eqref{mo2} and using the usual integration by parts, we have
\begin{equation}\label{eq5}
0=-\sum_{T\in {\cal T}_h} (\bu, \nabla q_0)+\sum_{T\in {\cal T}_h} \langle \bu\cdot\bn, q_0-q_b\rangle_{\pT},
\end{equation}
where we used $\sum_{T\in {\cal T}_h} \langle \bu\cdot\bn,  q_b\rangle_{\pT}=0$. 
Using \eqref{disgradient} and the usual integration by parts gives
\begin{equation} \label{beq}
\begin{split}
-b(\bQ_h\bu, q_h) =&-\sum_{T\in {\cal T}_h} (\bQ_0\bu, \nabla_w q_h)_T\\
=&\sum_{T\in {\cal T}_h} (q_0, \nabla\cdot(\bQ_0\bu))_T-\langle q_b, \bQ_0\bu\cdot\bn\rangle_{\pT}\\
=&\sum_{T\in {\cal T}_h} -(\nabla q_0,  \bQ_0\bu)_T+\langle q_0-q_b, \bQ_0\bu\cdot\bn\rangle_{\pT}\\
=&\sum_{T\in {\cal T}_h} -(\nabla q_0,  \bu)_T+\langle q_0-q_b, \bQ_0\bu\cdot\bn\rangle_{\pT}\\
=&\sum_{T\in {\cal T}_h}\langle q_0-q_b, (\bQ_0-I)\bu\cdot\bn\rangle_{\pT},
\end{split}
\end{equation}
where we used \eqref{eq5} on the last line. 

Subtracting \eqref{2} from the above equation completes the proof of 
\eqref{sehv2}.

This completes the proof of the lemma. 
\end{proof}

%
%The equations (\ref{sehv})-(\ref{sehv2}) are called {\em error
%equations} for the WG finite element scheme
%(\ref{32})-(\ref{2}).

\section{Error Estimates}\label{Section:error-estimates}
For any $\bv\in V_h^0$, we define the energy norm $\3bar \bv\3bar$ as follows 
\begin{equation}
\3bar \bv\3bar^2=\sum_{T\in {\cal T}_h} \| (\nabla \times)_w^2\bv\|_T^2+s_1(\bv, \bv).
\end{equation} 
It is easy to check that $\3bar\cdot \3bar$ is a semi-norm in $V_h^0$. We further introduce a norm in $V_h^0$; i.e.,
\begin{equation}
\3bar \bv\3bar_1 =\3bar \bv\3bar+\Big(\sum_{T\in {\cal T}_h} \|\nabla\cdot \bv_0\|_T^2\Big)^{\frac{1}{2}}+\Big(\sum_{e\in {\cal E}^0_h} h_T^{-1}\|\ljump \bv_0\cdot\bn\rjump \|_e^2\Big)^{\frac{1}{2}}.
\end{equation} 

For any $q\in W_h^0$,  we define the following norm 
$$
\3bar q\3bar_0=(s_2(q, q))^{\frac{1}{2}}.
$$

Recall that $\T_h$ is a shape-regular finite element partition of
the domain $\Omega$. For any $T\in\T_h$ and $\varphi\in H^{1}(T)$, the following trace inequality holds true \cite{wy3655}:
\begin{equation}\label{trace-inequality}
\|\varphi\|_{\pT}^2 \leq C
(h_T^{-1}\|\varphi\|_{T}^2+h_T\| \varphi\|_{1, T}^2).
\end{equation}
Furthermore, if $\varphi$ is a polynomial on $T$, the standard inverse inequality yields
\begin{equation}\label{trace}
\|\varphi\|_{\pT}^2 \leq Ch_T^{-1}\|\varphi\|_{T}^2.
\end{equation}
 
%  \begin{lemma}\cite{sun, monk, wgmaxwell} \label{lem1}
% Suppose $\bv\in [H^s(\Omega)]^3$, $\nabla\times\bv\in [H^s(\Omega)]^3$,  $(\nabla\times)^2\bv\in [H^{s-2}(\Omega)]^3$ and $q\in H^{s}(\Omega)$ with $1+\delta\leq s\leq k+1(\delta>0)$. Then, for $m=0, 1$ and $k\geq 2$, 
%  {\color{blue}\begin{equation*}
% \begin{split}
% \sum_{T\in {\cal T}_h}\|\bv-\bQ_0\bv\|_{m, T}^2 \leq & Ch^{2(s-m)}\|\bv\|_s^2,\\
% \sum_{T\in {\cal T}_h}\|\nabla\times(\bv-\bQ_0\bv)\|_{m, T}^2 \leq & Ch^{2(s-m)}\|\nabla\times\bv\|_s^2,\\
% \sum_{T\in {\cal T}_h}\|q- Q_0q\|_{m, T}^2 \leq & Ch^{2(s-m)}\|q\|_{s}^2,\\
%   \sum_{T\in {\cal T}_h}\|(I-{\cal Q}_h^{k-1}) (\nabla\times)^2\bv\|_{m, T}^2 \leq & Ch^{2(s-2-m)}\|(\nabla\times)^2\bv\|^2_{s-2}.   
% \end{split}
% \end{equation*}}
% \end{lemma}

\begin{lemma}\label{lem2}  Let $k\geq 1$, and   $s\in [1, k]$.
Suppose $\bu\in [H^{k+1}(\Omega)]^3$ and $(\nabla\times)^2\bu\in [H^k(\Omega)]^3$. Then, for $(\bv, q)\in V_h^0\times W_h^0$, the following estimates hold true; i.e.,
\begin{eqnarray}\label{error1}
|s_1(\bQ_h\bu, \bv)|&\leq& Ch^{s-1} \|\bu\|_{s+1} s_1(\bv, \bv)^{\frac{1}{2}},\\
|\ell_1(\bu, \bv)|&\leq&    Ch^{s-1}\|(\nabla\times)^2\bu\|_{s-1}s_1(\bv, \bv)^{\frac{1}{2}},\label{error2}\\
|\ell_2(\bu, q)| &\leq& Ch^{s-1}\|\bu\|_{s+1}\3bar q \3bar_0,\label{error3}\\
|s_2(Q_hp, q)|&=&0.\label{error4}
\end{eqnarray}
\end{lemma}
\begin{proof}
Using the Cauchy-Schwarz inequality, the trace inequality \eqref{trace-inequality},  gives
\begin{equation*}
\begin{split}
&|s_1(\bQ_h\bu, \bv)| = \Big| \sum_{T\in {\cal T}_h}h_T^{-3}\langle (\bQ_0\bu-\bQ_b\bu)\times\bn, (\bv_0-\bv_b)\times\bn\rangle_{\pT}\\&+ h_T^{-1} \langle \nabla\times\bQ_0\bu\times\bn-\bQ_n(\nabla\times\bu)\times\bn, \nabla\times\bv_0\times\bn-\bv_n\times\bn\rangle_{\pT} \Big|\\
% \leq &\Big| \sum_{T\in {\cal T}_h}h_T^{-3}\langle \bQ_0\bu- \bu, \bv_0-\bv_b\rangle_{\pT}\\&+ h_T^{-1} \langle \nabla\times\bQ_0\bu - \nabla\times\bu, \nabla\times\bv_0 -\bv_n \rangle_{\pT} \Big|\\
\leq & \{\big( \sum_{T\in {\cal T}_h}h_T^{-3}\| \bQ_0\bu- \bu\|^2_{\pT}\big)^{\frac{1}{2}}+  \big( \sum_{T\in {\cal T}_h}h_T^{-1}\|\nabla\times( \bQ_0\bu- \bu)\|^2_{\pT}\big)^{\frac{1}{2}}\}s_1(\bv, \bv)^{\frac{1}{2}}\\
\leq & \{\big( \sum_{T\in {\cal T}_h}h_T^{-4}\| \bQ_0\bu- \bu\|^2_{T}+h_T^{-2}\| \bQ_0\bu- \bu\|^2_{1, T}\big)^{\frac{1}{2}}\\
&+  \big( \sum_{T\in {\cal T}_h}h_T^{-2}\|\nabla\times( \bQ_0\bu- \bu)\|^2_{T}+\|\nabla\times( \bQ_0\bu- \bu)\|^2_{1,T}\big)^{\frac{1}{2}}\}s_1(\bv, \bv)^{\frac{1}{2}}\\
% \leq & C({\color{blue}h^{k-1}\|\nabla\times\bu\|_{k}}+h^{k-1}\|\bu\|_{k+1}) s_1(\bv, \bv)^{\frac{1}{2}}.
\leq & C h^{s-1}\|\bu\|_{s+1} s_1(\bv, \bv)^{\frac{1}{2}}.
\end{split}
\end{equation*}

Using the Cauchy-Schwarz inequality, the trace inequality \eqref{trace-inequality},  gives 
\begin{equation*}
\begin{split}
&\ell_1(\bu, \bv) \\=&\sum_{T\in{\cal T}_h}\langle (\bv_0-\bv_b)\times\bn, \nabla\times({\cal Q}_h^{k-2} -I)((\nabla\times)^2\bu)\rangle_{\partial T} \\&+\langle (\nabla \times \bv_0-\bv_n)\times\bn, ({\cal Q}_h^{k-2} -I)((\nabla\times)^2\bu)\rangle_{\partial T}
\\
\leq &\{\big(\sum_{T\in{\cal T}_h}h_T^3\|\nabla\times({\cal Q}_h^{k-2} -I)((\nabla\times)^2\bu)\|^2_{\partial T}\big)^{\frac{1}{2}} \\&+\big(\sum_{T\in{\cal T}_h}h_T\|({\cal Q}_h^{k-2} -I)((\nabla\times)^2\bu)\|^2_{\partial T}\big)^{\frac{1}{2}}\} s_1(\bv, \bv)^{\frac{1}{2}}\\
\leq &\{\big(\sum_{T\in{\cal T}_h}h_T^2\|\nabla\times({\cal Q}_h^{k-2} -I)((\nabla\times)^2\bu)\|^2_{T}+h_T^4\|\nabla\times({\cal Q}_h^{k-2} -I)((\nabla\times)^2\bu)\|^2_{1, T}\big)^{\frac{1}{2}} \\&+ \big(\sum_{T\in{\cal T}_h} \|({\cal Q}_h^{k-2} -I)((\nabla\times)^2\bu)\|^2_{T}+h_T^2\|({\cal Q}_h^{k-2} -I)((\nabla\times)^2\bu)\|^2_{1, T}\big)^{\frac{1}{2}}\} s_1(\bv, \bv)^{\frac{1}{2}}\\
\leq &  Ch^{s-1} \|(\nabla\times)^2 \bu\|_{s-1} s_1(\bv, \bv)^{\frac{1}{2}}.
\end{split}
\end{equation*}

Similarly, using the Cauchy-Schwarz inequality, the trace inequality \eqref{trace-inequality}  gives
\begin{equation*}
\begin{split}
\ell_2(\bu, q)&=\sum_{T\in{\cal T}_h}\langle q_0-q_b, (I-\bQ_0)\bu \cdot\bn\rangle_{\pT}\\
&\leq \Big(\sum_{T\in{\cal T}_h}  h_T^3\|q_0-q_b\|_{\pT}^2\Big)^{\frac{1}{2}}\Big(\sum_{T\in{\cal T}_h}h_T^{-3} \|(I-\bQ_0)\bu \cdot\bn\|_{\pT}^2\Big)^{\frac{1}{2}}\\
&\leq \Big(\sum_{T\in{\cal T}_h} h_T^{-4}\|(I-\bQ_0)\bu \cdot\bn\|_{T}^2+ h_T^{-2}\|(I-\bQ_0)\bu \cdot\bn\|_{1,T}^2\Big)^{\frac{1}{2}}\3bar q\3bar_0 \\
&\leq Ch^{s-1}\|\bu\|_{s+1}\3bar q\3bar_0.
  \end{split}
\end{equation*}

 Since $p=0$, it is easy to obtain $s_2(Q_hp, q)=0$. 
\end{proof}

\begin{theorem}\label{thm}
Let $k\geq 1$. Suppose that  $\bu\in [H^{k+1}(\Omega)]^3$. The following error estimate holds 
 \begin{equation}\label{estimate1}
\3bar \be_h\3bar+\3bar \epsilon_h \3bar_0 \leq Ch^{k-1} (\|\bu\|_{k+1}+\|(\nabla \times)^2\bu\|_{k-1}).
\end{equation} 
\end{theorem}
\begin{proof}
Letting $\bv_h=\be_h$ in (\ref{sehv}) and $q_h=\epsilon_h$ in (\ref{sehv2})  and adding the two equations, we have 
\begin{equation*}
\begin{split}
\3bar \be_h\3bar^2+\3bar \epsilon_h \3bar_0^2 &=s_1(\bQ_h\bu, \be_h)+s_2(Q_hp, \epsilon_h)+\ell_1(\bu, \be_h)-\ell_2(\bu, \epsilon_h).\\
\end{split}
\end{equation*} 
Using Lemma \ref{lem2} completes the proof of the theorem.
\end{proof}

\section{$L^2$ Error Estimates}
We consider an auxiliary problem of finding $(\bphi; \xi)$ such that
\begin{equation}\label{dual}
\begin{split}
(\nabla \times)^4 \bphi+\nabla \xi =&\be_0, \qquad \text{in}\ \Omega,\\
\nabla\cdot\bphi =& 0, \qquad\text{in} \ \Omega,\\
\bphi\times\bn=& 0, \qquad\text{on} \ \partial\Omega,\\
\nabla \times\bphi=& 0, \qquad\text{on} \ \partial\Omega,\\
\xi=& 0, \qquad\text{on} \ \partial\Omega. 
\end{split}
\end{equation}

% The weak formulation of the above problem is the same as \eqref{weakform} but with the right hand side $(\nabla \times)^2_{w, k} \be_h$ not necessarily satisfying divergence-free condition. The existence of a weak solution $(\bphi, \xi)$ can be obtained similarly.

Let ${t_0}=\min\{k, 3\}$. We assume the regularity property holds true in the sense that $\bphi$ and $\xi$ satisfy
 \begin{equation}\label{regu}
 \|\bphi\|_{t_0+1} +\|(\nabla\times)^2\bphi\|_{t_0-1}+\|\xi\|_{1}\leq C\| \be_0\|.
\end{equation}
 
 \begin{theorem}
 Let $k\geq 1$ and   ${t_0}=\min\{k, 3\}$. Suppose that  $\bu\in [H^{k+1}(\Omega)]^3$. The following estimate holds 
 \begin{equation}
      \|\be_0\|  \leq  Ch^{t_0+k-2} (\|\bu\|_{k+1}+\|(\nabla \times)^2\bu\|_{k-1}).
 \end{equation}
 In other words, we have a sub-optimal order of convergence for $k=1, 2$ and optimal order of convergence for $k\geq 3$.
 \end{theorem}
    
 \begin{proof}
 Using the usual integration by parts,  letting $\bu=\bphi$ and $\bv_h=\be_h$ in \eqref{eq1},  letting  $\bv_h=\bQ_h\bphi$ in \eqref{sehv},   letting  $\bu=\bphi$ and $q_h=\epsilon_h$ in \eqref{beq}, letting $q_h=Q_h\xi$ in \eqref{sehv2} and \eqref{l-2}, we have  
 \begin{equation}\label{dq}
     \begin{split}
& \|\be_0\|^2\\=&\sum_{T\in {\cal T}_h}((\nabla \times)^4 \bphi+\nabla \xi, \be_0)_T\\
 =&\sum_{T\in {\cal T}_h} ((\nabla \times)^2\bphi, (\nabla \times)^2\be_0)_T+\langle \nabla\times\be_0, \bn\times (\nabla \times)^2\bphi\rangle_{\partial T}\\
 &+\langle \be_0,\bn\times(\nabla \times)^3\bphi\rangle_{\partial T} +(\be_0, {\cal Q}^{k}_h \nabla  \xi)_T\\
=&\sum_{T\in {\cal T}_h}((\nabla\times)^2_w\bQ_h\bphi,  (\nabla\times)^2_w \be_h)_T  -\langle (\be_0-\be_b)\times\bn, \nabla\times{\cal Q}_h^{k-1} ((\nabla\times)^2\bphi)\rangle_{\partial T}\\&-\langle (\nabla \times \be_0-\be_n)\times\bn, {\cal Q}_h^{k-1} ((\nabla\times)^2\bphi)\rangle_{\partial T}
\\& +\langle \nabla\times\be_0, \bn\times (\nabla \times)^2\bphi\rangle_{\partial T}+\langle \be_0,\bn\times(\nabla \times)^3\bphi\rangle_{\partial T}  +(\nabla_w Q_h\xi, \be_0)_T 
\\=&-s_1(\be_h, \bQ_h\bphi)-b(\bQ_h\bphi, \epsilon_h)  + s_1(\bQ_h\bu, \bQ_h\bphi) +\ell_1(\bu, \bQ_h\bphi)  \\
&+\sum_{T\in {\cal T}_h}-\langle (\be_0-\be_b)\times\bn, \nabla\times({\cal Q}_h^{k-1}-I) ((\nabla\times)^2\bphi)\rangle_{\partial T}\\&-\langle (\nabla \times \be_0-\be_n)\times\bn,( {\cal Q}_h^{k-1}-I) ((\nabla\times)^2\bphi)\rangle_{\partial T} \\
&+
 s_2(\epsilon_h, Q_h\xi)-s_2(Q_hp, Q_h\xi)+\ell_2(\bu, Q_h\xi) 
\\=&-s_1(\be_h, \bQ_h\bphi)  + s_1(\bQ_h\bu, \bQ_h\bphi) +\ell_1(\bu, \bQ_h\bphi)  \\
&+\sum_{T\in {\cal T}_h}  \langle \epsilon_0-\epsilon_b, (\bQ_0-I)\bphi\cdot\bn\rangle_{\pT}\\&
-\langle (\be_0-\be_b)\times\bn, \nabla\times({\cal Q}_h^{k-1}-I) ((\nabla\times)^2\bphi)\rangle_{\partial T} \\&-\langle (\nabla \times \be_0-\be_n)\times\bn,( {\cal Q}_h^{k-1}-I) ((\nabla\times)^2\bphi)\rangle_{\partial T} \\
&+
 s_2(\epsilon_h, Q_h\xi)-s_2(Q_hp, Q_h\xi)+\ell_2(\bu, Q_h\xi),
\end{split}
\end{equation} 
where we used  $\be_b\times\bn=0$ and  $\be_n\times\bn=0$ on $\partial\Omega$.

% Note that 
% \begin{equation*}
%     \begin{split}
%  &\sum_{T\in {\cal T}_h} (\be_0, \nabla\xi-\nabla Q_0\xi)_T +\langle Q_0\xi-Q_b\xi, \be_0\cdot\bn\rangle_{
% \partial T} \\   
% =& \sum_{T\in {\cal T}_h} (\be_0, \nabla\xi)_T+( \nabla\be_0,  Q_0\xi)_T-\langle \be_0\cdot\bn,  Q_0\xi\rangle_{\partial T}+\langle Q_0\xi-Q_b\xi, \be_0\cdot\bn\rangle_{
% \partial T} \\
%  =& \sum_{T\in {\cal T}_h} -(\nabla\be_0,  \xi)_T+\langle \be_0\cdot\bn, \xi\rangle_{\partial T}+( \nabla\be_0,   \xi)_T -\langle  Q_b\xi, \be_0\cdot\bn\rangle_{
% \partial T} \\
% =&0,
% \end{split}
% \end{equation*}
% where we used the integration by parts. \eqref{dq1} can be simplified into
% \begin{equation}\label{dq}
%     \begin{split}
% &     \|\be_0\|^2\\    
%   =&-s_1(\be_h, \bQ_h\bphi)  + s_1(\bQ_h\bu, \bQ_h\bphi) +\ell_1(\bu, \bQ_h\bphi)  \\
% &+\sum_{T\in {\cal T}_h}  \langle \epsilon_0-\epsilon_b, (\bQ_0-I)\bphi\cdot\bn\rangle_{\pT}\\&
% -\langle (\be_0-\be_b)\times\bn, \nabla\times({\cal Q}_h^{k-1}-I) ((\nabla\times)^2\bphi)\rangle_{\partial T} \\&-\langle (\nabla \times \be_0-\be_n)\times\bn,( {\cal Q}_h^{k-1}-I) ((\nabla\times)^2\bphi)\rangle_{\partial T} \\
% &+
%  s_2(\epsilon_h, Q_h\xi)-s_2(Q_hp, Q_h\xi)+\ell_2(\bu, Q_h\xi).
%  \end{split}
% \end{equation}

Next, we shall estimate the terms on the last line of \eqref{dq} one by one.

 Recall that $t_0=\min\{k, 3\}$. Using \eqref{error1} with $\bv=\be_h$ and $\bu=\bphi$, we have
\begin{equation}\label{q1}
   |s_1(\be_h, \bQ_h\bphi)| \leq  Ch^{{t_0}-1} \|\bphi\|_{{t_0}+1} s_1(\be_h, \be_h)^{\frac{1}{2}}\leq Ch^{{t_0}-1} \|\bphi\|_{{t_0}+1} \3bar \be_h\3bar. 
\end{equation}
 
Using \eqref{error1} with $\bv=\bQ_h\bphi$, we have
\begin{equation}\label{q2}
  |s_1(\bQ_h\bu, \bQ_h\bphi)| \leq Ch^{k-1} \|\bu\|_{k+1} s_1(\bQ_h\bphi, \bQ_h\bphi)^{\frac{1}{2}}.
\end{equation}
Note that 
\begin{equation}\label{s1}
\begin{split}
  &s_1(\bQ_h\bphi, \bQ_h\bphi)^{\frac{1}{2}}\\ \leq& C\sum_{T\in {\cal T}_h} h_T^{-3}\| \bQ_0\bphi\times\bn-\bQ_b\bphi \times\bn\|_{\partial T}^2\\&+ h_T^{-1}\| \nabla\times \bQ_0\bphi\times\bn-\bQ_n(\nabla \times\bphi) \times\bn\|_{\partial T}^2\\
   \leq& C\sum_{T\in {\cal T}_h} h_T^{-4}\| \bQ_0\bphi - \bphi \|_{T}^2+h_T^{-2}\| \bQ_0\bphi - \bphi \|_{1,T}^2\\&+ h_T^{-2}\|   \bQ_0\bphi -  \bphi \|_{1, T}^2+ \|   \bQ_0\bphi -  \bphi \|_{2, T}^2\\
  \leq & Ch^{{t_0}-1}\|\bphi\|_{{t_0}+1}.
\end{split} 
\end{equation}
where we used trace inequality \eqref{trace-inequality}. Substituting \eqref{s1} into \eqref{q2} gives
\begin{equation}\label{q2_2}
    |s_1(\bQ_h\bu, \bQ_h\bphi)| \leq Ch^{k-1} \|\bu\|_{k+1}  h^{{t_0}-1}\|\bphi\|_{{t_0}+1}.
\end{equation}

Using \eqref{error2} with $\bv=\bQ_h\bphi$ and \eqref{s1}, we have
 \begin{equation}\label{q3}
    |\ell_1(\bu, \bQ_h\bphi)| \leq     Ch^{k-1}\|(\nabla\times)^2\bu\|_{k-1}s_1(\bQ_h\bphi, \bQ_h\bphi)^{\frac{1}{2}}\leq     Ch^{k-1}\|(\nabla\times)^2\bu\|_{k-1} h^{{t_0}-1}\|\bphi\|_{{t_0}+1}. 
\end{equation}

Using \eqref{error3} with $\bu=\bphi$ and $q_h=\epsilon_h$, we have
\begin{equation}\label{q4}
   |\sum_{T\in {\cal T}_h}   \langle \epsilon_0-\epsilon_b, (\bQ_0-I)\bphi\cdot\bn\rangle_{\pT} | \leq  Ch^{{t_0-1}}\|\bphi\|_{{t_0}+1}\3bar \epsilon_h \3bar_0,
\end{equation}

 Using \eqref{error2} with $\bu=\bphi$ and $\bv=\be_h$, we have
 \begin{equation}\label{q5}
 \begin{split}
  &   |\sum_{T\in {\cal T}_h}  \langle (\be_0-\be_b)\times\bn, \nabla\times({\cal Q}_h^{k-1}-I) ((\nabla\times)^2\bphi)\rangle_{\partial T}\\&+\langle (\nabla \times \be_0-\be_n)\times\bn,( {\cal Q}_h^{k-1}-I) ((\nabla\times)^2\bphi)\rangle_{\partial T}|\\
   \leq &    Ch^{{t_0}-1}\|(\nabla\times)^2 \bphi\|_{{t_0}-1}s_1(\be_h, \be_h)^{\frac{1}{2}} 
    \leq      Ch^{{t_0}-1}\|(\nabla\times)^2 \bphi\|_{{t_0}-1}\3bar \be_h\3bar. 
 \end{split}
 \end{equation}

% Using  Cauchy-Schwartz inequality, we have
% \begin{equation}
%     \begin{split}  
%   |\sum_{T\in {\cal T}_h} (\be_0, \nabla\xi-\nabla Q_0\xi)_T|&\leq  \big(\sum_{T\in {\cal T}_h} \|\be_0\|_T^2\big)^{\frac{1}{2}}\big(\sum_{T\in {\cal T}_h} \|\nabla\xi-\nabla Q_0\xi\|_T^2\big)^{\frac{1}{2}}\\
%   &\leq  \big(\sum_{T\in {\cal T}_h} \|\be_0\|_T^2\big)^{\frac{1}{2}}\big(\sum_{T\in {\cal T}_h} \| \xi-  Q_0\xi\|_{1,T}^2\big)^{\frac{1}{2}}\\
%   &\leq Ch^{t_0}\|\xi\|_{{t_0}+1}\|\be_0\|.
%     \end{split}
% \end{equation}

Using the Cauchy-Schwartz inequality and trace inequality \eqref{trace-inequality}, we have
\begin{equation}
    \begin{split}   
 s_2(\epsilon_h, Q_h\xi) =&\sum_{T\in {\cal T}_h} h_{T}^3 \langle \epsilon_0-\epsilon_b, Q_0\xi-Q_b\xi\rangle_{\partial T} \\
 \leq &C\3bar \epsilon_h\3bar_0 \Big(\sum_{T\in {\cal T}_h} h_{T}^3 \|Q_0\xi- \xi\|_{\partial T}^2\big)^{\frac{1}{2}}\\
  \leq &C\3bar \epsilon_h\3bar_0 \Big(\sum_{T\in {\cal T}_h}h_T^{2}  \|Q_0\xi- \xi\|_{  T}^2+ h_T^{4}\|Q_0\xi- \xi\|_{1,  T}^2\big)^{\frac{1}{2}}\\
  \leq &Ch^{2}\|\xi\|_{1}\3bar \epsilon_h\3bar_0.
     \end{split}
\end{equation}
 
 Using \eqref{error4} with $q=Q_h\xi$, we have
 \begin{equation}
     s_2(Q_hp, Q_h\xi)=0.
 \end{equation}

 Using \eqref{error3} with $q=Q_h\xi$ and the trace inequality \eqref{trace-inequality},  we have
 \begin{equation}\label{last}
     \begin{split}
         \ell_2(\bu, Q_h\xi)&\leq Ch^{k-1}\|\bu\|_{k+1} \3bar Q_h\xi \3bar_0\\
         &\leq Ch^{k-1}\|\bu\|_{k+1}  \big(\sum_{T\in {\cal T}_h}h_T^3 \|Q_0\xi-Q_b\xi\|_{\partial T}^2\big)^{\frac{1}{2}}\\
           &\leq Ch^{k-1}\|\bu\|_{k+1}  \big(\sum_{T\in {\cal T}_h}h_T ^3\|Q_0\xi- \xi\|_{\partial T}^2\big)^{\frac{1}{2}}  \\
           &\leq Ch^{k-1}\|\bu\|_{k+1}  \big(\sum_{T\in {\cal T}_h}h_T^{2}\|Q_0\xi- \xi\|_{ T}^2+h_T ^4\|Q_0\xi- \xi\|_{1, T}^2\big)^{\frac{1}{2}}\\
            &\leq Ch^{k-1}\|\bu\|_{k+1}  h^{2}\|\xi\|_{1}.
     \end{split}
    \end{equation}
    
%   Using Cauchy-Schwartz inequality and trace inequalities \eqref{trace} and \eqref{trace-inequality} gives
%     \begin{equation}\label{last}
%      \begin{split}
%   &\qquad | \sum_{T\in {\cal T}_h} \langle Q_0\xi-Q_b\xi, \be_0\cdot\bn\rangle_{
% \partial T}|\\&\leq C \big(\sum_{T\in {\cal T}_h}\|Q_0\xi-Q_b\xi\|_{\partial T}^2\big)^{\frac{1}{2}} \big(\sum_{T\in {\cal T}_h}\|\be_0\cdot\bn\|_{\partial T}^2\big)^{\frac{1}{2}}\\
% &\leq C\big(\sum_{T\in {\cal T}_h}h_T^{-1}\|Q_0\xi- \xi\|^2_{ T}+h_T \|Q_0\xi- \xi\|^2_{1, T}\big)^{\frac{1}{2}} \big(\sum_{T\in {\cal T}_h}h_T^{-1}\|\be_0 \|_{  T}^2\big)^{\frac{1}{2}}\\
% &\leq Ch^{{t_0}}\|\xi\|_{{t_0}+1}\|\be_0\|.
%   \end{split} 
%   \end{equation}
      
      Substituting \eqref{q1}-\eqref{last} into 
      \eqref{dq} and using the regularity assumption \eqref{regu}  and the error estimate \eqref{estimate1}   gives
      \begin{equation*}
          \begin{split}
              \|\be_0\|^2\leq &C h^{{t_0}-1} \|\bphi\|_{{t_0}+1} \3bar \be_h\3bar+ C h^{k-1}\|\bu\|_{k+1}h^{{t_0}-1} \|\bphi\|_{{t_0}+1}\\
              &+ C h^{k-1}\|(\nabla\times)^2\bu\|_{k-1}h^{{t_0}-1} \|\bphi\|_{{t_0}+1}+C h^{{t_0}-1} \|\bphi\|_{{t_0}+1} \3bar \epsilon_h\3bar_0\\
&+              Ch^{{t_0}-1}\|(\nabla\times)^2 \bphi\|_{{t_0-1}}\3bar \be_h\3bar+ 
Ch^{2}\|\xi\|_{1}\3bar \epsilon_h\3bar_0 \\&+Ch^{k-1}\|\bu\|_{k+1}  h^{2}\|\xi\|_{1},
          \end{split}
      \end{equation*}
      which yields
      \begin{equation*}
          \begin{split}
        \|\be_0\|^2 \leq  Ch^{t_0+k-2} (\|\bu\|_{k+1}+\|(\nabla \times)^2\bu\|_{k-1})\|\be_0\|.
          \end{split}
      \end{equation*}   
      This completes the proof of the theorem.
 \end{proof}

\section{Numerical tests}
In this section, we present some numerical results for the WG finite
element method for solving the quad-curl problem  analyzed in the  previous sections. To this end, we shall solve the following quad-curl problem with non-homogeneous boundary conditions on an unit cube domain $\Omega=(0,1)^3$: Find an known $\bu$ such that 
\begin{equation}\label{s-1} 
\begin{split}
 (\nabla \times)^4 \bu=&\bf, \qquad \text{in}\quad \Omega,\\
 \nabla\cdot\bu=&0, \qquad \text{in}\quad \Omega,\\
 \bu\times\bn=&{\mathbf g}_1, \qquad \text{on}\quad \partial\Omega,\\
\nabla\times\bu\times\bn=&{\mathbf g}_2, \qquad \text{on}\quad \partial\Omega,
 \end{split}
\end{equation} where $\bf$, ${\mathbf g}_1$ and ${\mathbf g}_2$ are calculated by the exact solution 
\begin{align*} \bu = \begin{pmatrix} -2 x^2 y^2 z \\
                                                       2 x^2 y^3 z \\
                                                     -x y^2 z^2 (3 x - 2) \end{pmatrix}.
\end{align*} 

% In order to get possible superconvergent solutions, we revise Algorithm \ref{a-1} slightly in numerical computation:
%   Find $(\bu_h; p_h)\in V_h^0 \times W_{h}^0$ (defined in \eqref{V0} and \eqref{W0} respectively), such that
%   \begin{align}\label{a-2}\begin{aligned}
%   s_1(\bu_h, \bv_h)+a(\bu_h, \bv_h)+h^{-2} b(\bv_h, p_h)&= (\bf, \bv_0), & \forall \bv_h &\in V_{h}^0,\\
%   \qquad h^{-2} b(\bu_h, q_h)  \qquad -h^{-4}s_2(p_h, q_h) &= 0,\qquad &   \forall q_h &\in W_h^0, \end{aligned}
%   \end{align} where $a(\bu_h, \bv_h)$, $b(\bv_h, p_h)$, $s_1(\bu_h, \bv_h)$ and $s_2(p_h, q_h)$ are defined
%     in \eqref{EQ:local-stabilizer}--\eqref{EQ:local-bterm} respectively. 
% Note that the weak gradient operator and the weak curl-curl operator are
%   defined in (\ref{disgradient}) and \eqref{discurlcurl} respectively; i.e.,  
%   $\nabla_{w, k, T} \; p_h $ and  $(\nabla \times) ^2_{w, k-1, T} \bu_h$.

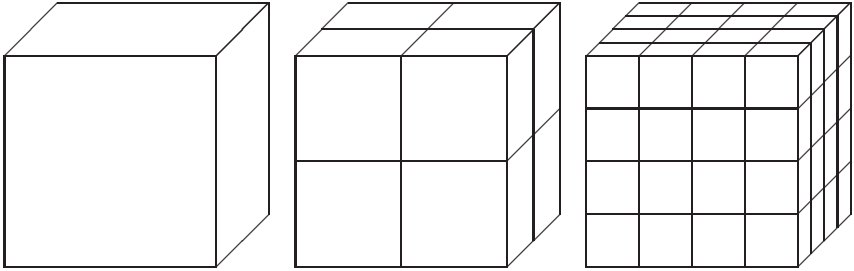
\begin{figure}[ht]
\begin{center}
 \setlength\unitlength{1pt}
    \begin{picture}(320,118)(0,3)
    \put(0,0){\begin{picture}(110,110)(0,0)
       \multiput(0,0)(80,0){2}{\line(0,1){80}}  \multiput(0,0)(0,80){2}{\line(1,0){80}}
       \multiput(0,80)(80,0){2}{\line(1,1){20}} \multiput(0,80)(20,20){2}{\line(1,0){80}}
       \multiput(80,0)(0,80){2}{\line(1,1){20}}  \multiput(80,0)(20,20){2}{\line(0,1){80}}
   % \put(80,0){\line(-1,1){80}}  \put(80,0){\line(1,5){20}}\put(80,80){\line(-3,1){60}}
      \end{picture}}
    \put(110,0){\begin{picture}(110,110)(0,0)
       \multiput(0,0)(40,0){3}{\line(0,1){80}}  \multiput(0,0)(0,40){3}{\line(1,0){80}}
       \multiput(0,80)(40,0){3}{\line(1,1){20}} \multiput(0,80)(10,10){3}{\line(1,0){80}}
       \multiput(80,0)(0,40){3}{\line(1,1){20}}  \multiput(80,0)(10,10){3}{\line(0,1){80}}
   % \put(80,0){\line(-1,1){80}}  \put(80,0){\line(1,5){20}}\put(80,80){\line(-3,1){60}}
   %    \multiput(40,0)(40,40){2}{\line(-1,1){40}}
       %  \multiput(80,40)(10,-30){2}{\line(1,5){10}}
       %  \multiput(40,80)(50,10){2}{\line(-3,1){30}}
      \end{picture}}
    \put(220,0){\begin{picture}(110,110)(0,0)
       \multiput(0,0)(20,0){5}{\line(0,1){80}}  \multiput(0,0)(0,20){5}{\line(1,0){80}}
       \multiput(0,80)(20,0){5}{\line(1,1){20}} \multiput(0,80)(5,5){5}{\line(1,0){80}}
       \multiput(80,0)(0,20){5}{\line(1,1){20}}  \multiput(80,0)(5,5){5}{\line(0,1){80}}
  %   \put(80,0){\line(-1,1){80}}  \put(80,0){\line(1,5){20}}\put(80,80){\line(-3,1){60}}
  %   \multiput(40,0)(40,40){2}{\line(-1,1){40}}
       %  \multiput(80,40)(10,-30){2}{\line(1,5){10}}
       %  \multiput(40,80)(50,10){2}{\line(-3,1){30}}

   %  \multiput(20,0)(60,60){2}{\line(-1,1){20}}   \multiput(60,0)(20,20){2}{\line(-1,1){60}}
       %  \multiput(80,60)(15,-45){2}{\line(1,5){5}} \multiput(80,20)(5,-15){2}{\line(1,5){15}}
       %  \multiput(20,80)(75,15){2}{\line(-3,1){15}}\multiput(60,80)(25,5){2}{\line(-3,1){45}}
      \end{picture}}

    \end{picture}
    \end{center}
\caption{  The first three levels of uniform cubic grids used in Table \ref{t-1}. }
\label{g-1}
\end{figure}

We first compute the solution of \eqref{s-1} by the $P_k$ weak Galerkin finite element method \eqref{32}-\eqref{2} on  uniform cubic grids shown in
   Figure \ref{g-1}.
For simplicity of notations, we denote the WG finite element solution $(\bu_h; p_h)$  by $\{P_k, P_k, P_{k-1}\}$-$P_{k-1}$ with $\{P_k, P_k\}$-$P_k$.
In Table \ref{t-1},  we list the errors in various norms and the computed orders of convergence for $P_2$, $P_3$, $P_4$ and $P_5$
  finite element solutions on uniform cubic grids. It seems we do have one order superconvergence in most cases in Table \ref{t-1}.

\begin{table}[ht]
  \centering   \renewcommand{\arraystretch}{1.05}
  \caption{ Error profiles and convergence rates on uniform cubic grids shown in Figure \ref{g-1} for \eqref{s-1}. }
\label{t-1}
\begin{tabular}{c|cc|cc|cc}
\hline
level & $\|\bQ_h  \bu-  \bu_h \| $  &rate &  $\3bar \bQ_h \bu- \bu_h \3bar $ &rate & $\|  p_h \| $  &rate  \\
\hline
 &\multicolumn{6}{c}{by the $\{P_2, P_2, P_1\}$-$P_1$ with $\{P_2, P_2\}$-$P_2$ WG method} \\ \hline
 2&  0.5263E+00&  3.8&  0.4078E+01&  1.7&  0.1742E-01&  3.9 \\
 3&  0.3345E-01&  4.0&  0.1237E+01&  1.7&  0.1947E-02&  3.2 \\
 4&  0.2151E-02&  4.0&  0.3739E+00&  1.7&  0.2231E-03&  3.1 \\
\hline
 &\multicolumn{6}{c}{by the $\{P_3, P_3, P_2\}$-$P_2$ with $\{P_3, P_3\}$-$P_3$ WG method} \\ \hline
 2&  0.8522E-01&  5.3&  0.1512E+01&  2.2&  0.1117E-01&  4.0 \\
 3&  0.2190E-02&  5.3&  0.2365E+00&  2.7&  0.4298E-03&  4.7 \\
 4&  0.6360E-04&  5.1&  0.3444E-01&  2.8&  0.1841E-04&  4.5 \\
 \hline
 &\multicolumn{6}{c}{by the $\{P_4, P_4, P_3\}$-$P_3$ with $\{P_4, P_4\}$-$P_4$ WG method} \\ \hline
 2&  0.9472E-02&  6.3&  0.4145E+00&  3.4&  0.1476E-02&  4.2 \\
 3&  0.1841E-03&  5.7&  0.3102E-01&  3.7&  0.1760E-04&  6.4 \\
 4&  0.2992E-05&  5.9&  0.2162E-02&  3.8&  0.3430E-06&  5.7 \\
\hline
 &\multicolumn{6}{c}{by the $\{P_5, P_5, P_4\}$-$P_4$  with $\{P_5, P_5\}$-$P_5$ WG method} \\ \hline
 1&  0.1010E-02&  0.0&  0.5724E-01&  0.0&  0.3493E-03&  0.0 \\
 2&  0.1684E-04&  5.9&  0.3565E-02&  4.0&  0.5073E-05&  6.1 \\
 3&  0.2580E-06&  6.0&  0.2204E-03&  4.0&  0.9438E-07&  5.7 \\
 \hline
\end{tabular}%
\end{table}%

Next we compute the solution of \eqref{s-1} again by the $P_k$ weak Galerkin finite element method but on uniform tetrahedral grids shown in
   Figure \ref{g-2}. 
In Table \ref{t-2},  we list the errors in various norms and the computed orders of convergence for $P_2$, $P_3$ and $P_4$ 
  finite element solutions on uniform tetrahedral grids. It seems we do have one order superconvergence in most cases in Table \ref{t-2}.

\begin{figure}[ht]
\begin{center}
 \setlength\unitlength{1pt}
    \begin{picture}(320,118)(0,3)
    \put(0,0){\begin{picture}(110,110)(0,0)
       \multiput(0,0)(80,0){2}{\line(0,1){80}}  \multiput(0,0)(0,80){2}{\line(1,0){80}}
       \multiput(0,80)(80,0){2}{\line(1,1){20}} \multiput(0,80)(20,20){2}{\line(1,0){80}}
       \multiput(80,0)(0,80){2}{\line(1,1){20}}  \multiput(80,0)(20,20){2}{\line(0,1){80}}
    \put(80,0){\line(-1,1){80}} \put(80,0){\line(1,5){20}}\put(80,80){\line(-3,1){60}}
      \end{picture}}
    \put(110,0){\begin{picture}(110,110)(0,0)
       \multiput(0,0)(40,0){3}{\line(0,1){80}}  \multiput(0,0)(0,40){3}{\line(1,0){80}}
       \multiput(0,80)(40,0){3}{\line(1,1){20}} \multiput(0,80)(10,10){3}{\line(1,0){80}}
       \multiput(80,0)(0,40){3}{\line(1,1){20}}  \multiput(80,0)(10,10){3}{\line(0,1){80}}
    \put(80,0){\line(-1,1){80}}  \put(80,0){\line(1,5){20}}\put(80,80){\line(-3,1){60}}
       \multiput(40,0)(40,40){2}{\line(-1,1){40}}
         \multiput(80,40)(10,-30){2}{\line(1,5){10}}
        \multiput(40,80)(50,10){2}{\line(-3,1){30}}
      \end{picture}}
    \put(220,0){\begin{picture}(110,110)(0,0)
       \multiput(0,0)(20,0){5}{\line(0,1){80}}  \multiput(0,0)(0,20){5}{\line(1,0){80}}
       \multiput(0,80)(20,0){5}{\line(1,1){20}} \multiput(0,80)(5,5){5}{\line(1,0){80}}
       \multiput(80,0)(0,20){5}{\line(1,1){20}}  \multiput(80,0)(5,5){5}{\line(0,1){80}}
    \put(80,0){\line(-1,1){80}}  \put(80,0){\line(1,5){20}}\put(80,80){\line(-3,1){60}}
       \multiput(40,0)(40,40){2}{\line(-1,1){40}}
       \multiput(80,40)(10,-30){2}{\line(1,5){10}}
        \multiput(40,80)(50,10){2}{\line(-3,1){30}}

       \multiput(20,0)(60,60){2}{\line(-1,1){20}}   \multiput(60,0)(20,20){2}{\line(-1,1){60}}
         \multiput(80,60)(15,-45){2}{\line(1,5){5}} \multiput(80,20)(5,-15){2}{\line(1,5){15}}
          \multiput(20,80)(75,15){2}{\line(-3,1){15}}\multiput(60,80)(25,5){2}{\line(-3,1){45}}
      \end{picture}}

    \end{picture}
    \end{center}
\caption{  The first three levels of uniform tetrahedral grids used in Table \ref{t-2}. }
\label{g-2}
\end{figure}
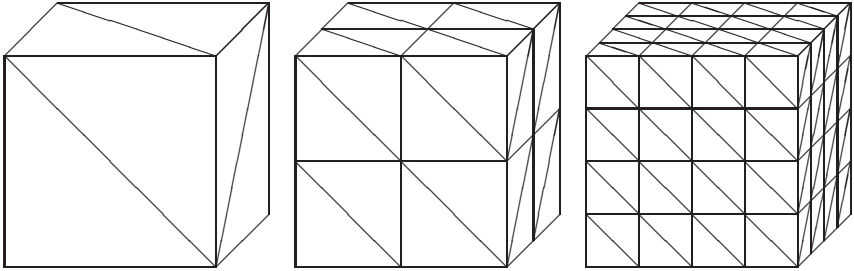

\begin{table}[ht]
  \centering   \renewcommand{\arraystretch}{1.05}
  \caption{ Error profiles and convergence rates on tetrahedral grids shown in Figure \ref{g-2} for \eqref{s-1}. }
\label{t-2}
\begin{tabular}{c|cc|cc|cc}
\hline
level & $\|\bQ_h  \bu-  \bu_h \| $  &rate &  $\3bar \bQ_h \bu- \bu_h \3bar $ &rate & $\|  p_h \| $  &rate  \\
\hline
 &\multicolumn{6}{c}{by the $\{P_2, P_2, P_1\}$-$P_1$ with $\{P_2, P_2\}$-$P_2$ WG method} \\ \hline
 2&   0.444E+00&3.9&   0.357E+01&1.8&   0.288E-01&3.8 \\
 3&   0.279E-01&4.0&   0.117E+01&1.6&   0.193E-02&3.9 \\
 4&   0.172E-02&4.0&   0.451E+00&1.4&   0.304E-03&2.7 \\
\hline
 &\multicolumn{6}{c}{by the $\{P_3, P_3, P_2\}$-$P_2$ with $\{P_3, P_3\}$-$P_3$ WG method} \\ \hline
 1&   0.187E+01&0.0&   0.480E+01&0.0&   0.181E+00&0.0 \\
 2&   0.551E-01&5.1&   0.724E+00&2.7&   0.764E-02&4.6 \\
 3&   0.163E-02&5.1&   0.128E+00&2.5&   0.257E-03&4.9 \\
 \hline
 &\multicolumn{6}{c}{by the $\{P_4, P_4, P_3\}$-$P_3$ with $\{P_4, P_4\}$-$P_4$ WG method} \\ \hline
 1&   0.265E+00&0.0&   0.888E+00&0.0&   0.308E-01&0.0 \\
 2&   0.411E-02&6.0&   0.738E-01&3.6&   0.632E-03&5.6 \\
 3&   0.693E-04&5.9&   0.721E-02&3.4&   0.139E-04&5.5 \\
\hline 
\end{tabular}%
\end{table}%

\long\def\myskip#1{}
\myskip{
\begin{verbatim}
 u0 L2 | u a |w l2 |id,ir            2           0
 1&  0.7541E+01&  0.0&  0.1336E+02&  0.0&  0.2588E+00&  0.0 \\
 2&  0.5316E+00&  3.8&  0.4103E+01&  1.7&  0.2746E-01&  3.2 \\
 3&  0.3379E-01&  4.0&  0.1241E+01&  1.7&  0.8302E-02&  1.7 \\
 4&  0.2188E-02&  3.9&  0.3744E+00&  1.7&  0.1795E-02&  2.2 \\
 5&  0.1508E-03&  3.9&  0.1154E+00&  1.7&  0.4554E-03&  2.0 \\


  u0 L2 | u a |w l2 |id,ir            2           0    (original s1*h wg*h**2)  s1*h wg*h**0 
 1&  0.7541E+01&  0.0&  0.1336E+02&  0.0&  0.2588E+00&  0.0 \\
 2&  0.5130E+00&  3.9&  0.4020E+01&  1.7&  0.4319E-01&  2.6 \\
 3&  0.3192E-01&  4.0&  0.1212E+01&  1.7&  0.5772E-02&  2.9 \\
 4&  0.2019E-02&  4.0&  0.3655E+00&  1.7&  0.7854E-03&  2.9 \\

  u0 L2 | u a |w l2 |id,ir            2           0      (original s1*h wg*h**2)  s1/h wg*h**0
 1&  0.7541E+01&  0.0&  0.1336E+02&  0.0&  0.2588E+00&  0.0 \\
 2&  0.5263E+00&  3.8&  0.4078E+01&  1.7&  0.1742E-01&  3.9 \\
 3&  0.3345E-01&  4.0&  0.1237E+01&  1.7&  0.1947E-02&  3.2 \\
 4&  0.2151E-02&  4.0&  0.3739E+00&  1.7&  0.2231E-03&  3.1 \\


  u0 L2 | u a |w l2 |id,ir            3           0
 1&  0.3473E+01&  0.0&  0.6987E+01&  0.0&  0.3286E+00&  0.0 \\
 2&  0.8914E-01&  5.3&  0.1492E+01&  2.2&  0.4130E-01&  3.0 \\
 3&  0.3030E-02&  4.9&  0.2343E+00&  2.7&  0.2288E-02&  4.2 \\
 4&  0.1579E-03&  4.3&  0.3418E-01&  2.8&  0.1169E-03&  4.3 \\

  u0 L2 | u a |w l2 |id,ir            3           0  (original s1*h wg*h**2)  s1*h**-1 wg*h**0  
 1&  0.3451E+01&  0.0&  0.7142E+01&  0.0&  0.1809E+00&  0.0 \\
 2&  0.8522E-01&  5.3&  0.1512E+01&  2.2&  0.1117E-01&  4.0 \\
 3&  0.2190E-02&  5.3&  0.2365E+00&  2.7&  0.4298E-03&  4.7 \\
 4&  0.6360E-04&  5.1&  0.3444E-01&  2.8&  0.1841E-04&  4.5 \\

  u0 L2 | u a |w l2 |id,ir            4           0
 1&  0.6412E+00&  0.0&  0.4487E+01&  0.0&  0.3939E-01&  0.0 \\
 2&  0.7991E-02&  6.3&  0.4145E+00&  3.4&  0.3765E-02&  3.4 \\
 3&  0.1110E-03&  6.2&  0.3104E-01&  3.7&  0.6406E-04&  5.9 \\
 4&  0.1579E-05&  6.1&  0.2163E-02&  3.8&  0.2198E-05&  4.9 \\


  u0 L2 | u a |w l2 |id,ir            4           0  (original s1*h wg*h**2)  s1*h**1 wg*h**0
 1&  0.7502E+00&  0.0&  0.4496E+01&  0.0&  0.2665E-01&  0.0 \\
 2&  0.8880E-02&  6.4&  0.4139E+00&  3.4&  0.2965E-02&  3.2 \\
 3&  0.1678E-03&  5.7&  0.3094E-01&  3.7&  0.9880E-04&  4.9 \\
 4&  0.2876E-05&  5.9&  0.2156E-02&  3.8&  0.4621E-05&  4.4 \\


  u0 L2 | u a |w l2 |id,ir            4           0  (original s1*h wg*h**2)  s1*h**-1 wg*h**0
 1&  0.7502E+00&  0.0&  0.4496E+01&  0.0&  0.2665E-01&  0.0 \\
 2&  0.9472E-02&  6.3&  0.4145E+00&  3.4&  0.1476E-02&  4.2 \\
 3&  0.1841E-03&  5.7&  0.3102E-01&  3.7&  0.1760E-04&  6.4 \\
 4&  0.2992E-05&  5.9&  0.2162E-02&  3.8&  0.3430E-06&  5.7 \\


  u0 L2 | u a |w l2 |id,ir            5           0
 1&  0.1010E-02&  0.0&  0.5724E-01&  0.0&  0.3493E-03&  0.0 \\
 2&  0.1824E-04&  5.8&  0.3566E-02&  4.0&  0.1333E-04&  4.7 \\
 3&  0.2938E-06&  6.0&  0.2205E-03&  4.0&  0.9786E-06&  3.8 \\


  u0 L2 | u a |w l2 |id,ir            5    (original s1*h wg*h**2)  s1*h**-1 wg*h**0
 1&  0.1010E-02&  0.0&  0.5724E-01&  0.0&  0.3493E-03&  0.0 \\
 2&  0.1684E-04&  5.9&  0.3565E-02&  4.0&  0.5073E-05&  6.1 \\

  u0 L2 | u a |w l2 |id,ir            5           0 s1*h**-1 wg*h**0
 1&  0.1010E-02&  0.0&  0.5724E-01&  0.0&  0.3493E-03&  0.0 \\
 2&  0.1684E-04&  5.9&  0.3565E-02&  4.0&  0.5073E-05&  6.1 \\
 3&  0.2580E-06&  6.0&  0.2204E-03&  4.0&  0.9438E-07&  5.7 \\
 

----- tri below.
  u0 L2 | u a |w l2 |id            2
 1&   0.648E+01&0.0&   0.123E+02&0.0&   0.414E+00&0.0 \\
 2&   0.418E+00&4.0&   0.346E+01&1.8&   0.301E-01&3.8 \\
 3&   0.248E-01&4.1&   0.113E+01&1.6&   0.206E-02&3.9 \\
 4&   0.147E-02&4.1&   0.442E+00&1.4&   0.146E-03&3.8 \\

   (original s1*h wg*h**2)  s1/h wg*h**0
                                              
  u0 L2 | u a |w l2 |id             2 s1/h wg*h**0
 1&   0.648E+01&0.0&   0.123E+02&0.0&   0.414E+00&0.0 \\
 2&   0.444E+00&3.9&   0.357E+01&1.8&   0.288E-01&3.8 \\
 3&   0.279E-01&4.0&   0.117E+01&1.6&   0.193E-02&3.9 \\
  u0 L2 | u a |w l2 |id             2
 1&   0.648E+01&0.0&   0.123E+02&0.0&   0.414E+00&0.0 \\
 2&   0.444E+00&3.9&   0.357E+01&1.8&   0.288E-01&3.8 \\
 3&   0.279E-01&4.0&   0.117E+01&1.6&   0.193E-02&3.9 \\
 4&   0.172E-02&4.0&   0.451E+00&1.4&   0.304E-03&2.7 \\

  u0 L2 | u a |w l2 |id             2 s1*h wg* h**2;  (standard)
 1&   0.000E+00&0.0&   0.000E+00&0.0&   0.000E+00&0.0 \\
 2&   0.456E+00&0.0&   0.363E+01&0.0&   0.549E-01&0.0 \\
 3&   0.287E-01&4.0&   0.117E+01&1.6&   0.147E-01&1.9 \\


  u0 L2 | u a |w l2 |id             2 
 1&   0.648E+01&0.0&   0.123E+02&0.0&   0.414E+00&0.0 \\
 2&   0.456E+00&3.8&   0.363E+01&1.8&   0.137E-01&4.9 \\
 3&   0.287E-01&4.0&   0.117E+01&1.6&   0.921E-03&3.9 \\

  u0 L2 | u a |w l2 |id             3
 1&   0.187E+01&0.0&   0.480E+01&0.0&   0.181E+00&0.0 \\
 2&   0.551E-01&5.1&   0.724E+00&2.7&   0.764E-02&4.6 \\
 3&   0.163E-02&5.1&   0.128E+00&2.5&   0.257E-03&4.9 \\


  u0 L2 | u a |w l2 |id,ir            3
 1&   0.187E+01&0.0&   0.480E+01&0.0&   0.181E+00&0.0 \\
 2&   0.631E-01&4.9&   0.765E+00&2.6&   0.747E-02&4.6 \\
 3&   0.195E-02&5.0&   0.133E+00&2.5&   0.133E-02&2.5 \\
 4&   0.605E-04&5.0&   0.269E-01&2.3&   0.474E-03&1.5 \\


  u0 L2 | u a |w l2 |id             4
 1&   0.265E+00&0.0&   0.888E+00&0.0&   0.308E-01&0.0 \\
 2&   0.377E-02&6.1&   0.718E-01&3.6&   0.645E-03&5.6 \\
 3&   0.590E-04&6.0&   0.704E-02&3.4&   0.127E-04&5.7 \\


  u0 L2 | u a |w l2 |id             4
 1&   0.265E+00&0.0&   0.888E+00&0.0&   0.308E-01&0.0 \\
 2&   0.411E-02&6.0&   0.738E-01&3.6&   0.632E-03&5.6 \\
 3&   0.693E-04&5.9&   0.721E-02&3.4&   0.139E-04&5.5 \\

\end{verbatim}
}

\bibliographystyle{abbrv}
\bibliography{Ref}

\end{document}